\documentclass[12pt,twoside]{amsart}
\usepackage{amsmath}
\usepackage{amsthm}
\usepackage{amsfonts}
\usepackage{amssymb}
\usepackage{latexsym}
\usepackage{mathrsfs}
\usepackage{amsmath}
\usepackage{amsthm}
\usepackage{amsfonts}
\usepackage{amssymb}
\usepackage{latexsym}
\usepackage{geometry}
\usepackage{dsfont}
\usepackage[dvips]{graphicx}
\usepackage{color}
\usepackage[all]{xy}

\date{}
\allowdisplaybreaks[4] \footskip=15pt
\renewcommand{\uppercasenonmath}[1]{}

\usepackage{graphicx,amssymb}
\usepackage[all]{xy}
\usepackage{amsmath}

\allowdisplaybreaks
\usepackage{amsthm}
\usepackage{color}

\theoremstyle{plain}
\newtheorem{theorem}{Theorem}[section]
\newtheorem{proposition}[theorem]{Proposition}

\newtheorem{corollary}[theorem]{Corollary}
\theoremstyle{definition}

\newtheorem{definition}[theorem]{Definition}
\newtheorem{question}[theorem]{Open Question}
\theoremstyle{definition}
\newtheorem*{acknowledgement}{Acknowledgement}

\theoremstyle{remark}
\newtheorem{remark}[theorem]{Remark}

\newcommand{\pf}{\noindent\begin {proof}}
\newcommand{\epf}{\end{proof}}

\newcommand{\Ker}{\mbox{\rm Ker}}
\newcommand{\Ext}{\mbox{\rm Ext}}
\newcommand{\Hom}{\mbox{\rm Hom}}
\newcommand{\Tor}{\mbox{\rm Tor}}

\newcommand{\Prufer}{Pr\"{u}fer}

\def\GV{{\rm GV}}

\def\Hom{{\rm Hom}}
\def\Ext{{\rm Ext}}
\def\Tor{{\rm Tor}}

\def\fkm{{\frak m}}

\def\ker{{\rm ker}}
\def\Ker{{\rm Ker}}

\def\Nil{{\rm Nil}}
\def\NN{{\rm NN}}

\def\NP{{\rm NP}}
\def\Z{{\rm Z}}

\def\T{{\rm T}}
\def\GV{{\rm GV}}
\def\Max{{\rm Max}}
\def\DW{{\rm DW}}
\def\ZN{{\rm ZN}}

\def\PvMR{{\rm PvMR}}
\def\PvMD{{\rm PvMD}}
\def\SM{{\rm SM}}
\def\Krull{{\rm Krull}}
\def\p{{\frak p}}

\begin{document}
\begin{center}
{\large  \bf On strongly $\phi$-$w$-Flat modules  and Their Homological Dimensions}

\vspace{0.5cm}  Wei Qi$^{a}$, Xiaolei Zhang$^{a}$,   Wei Zhao$^{b}$\\
%\bigskip

{\footnotesize a.\  School of Mathematics and Statistics, Shandong University of Technology, Zibo 255000, China\\

b.\ School of Mathematics, ABa Teachers University, Wenchuan 623002, China

E-mail: zxlrghj@163.com\\}
\end{center}

\bigskip
\centerline { \bf  Abstract}
\bigskip
\leftskip10truemm \rightskip10truemm \noindent

In this paper,  we introduce and study the notion of strongly $\phi$-$w$-flat modules. The $\phi$-$w$-weak global dimension $\phi$-$w$-w.gl.dim$(R)$ of an $\NP$-ring $R$ is also introduced and studied. We characterize $\phi$-\Prufer\ $v$-multiplication strongly $\phi$-rings in terms of strongly $\phi$-$w$-flat modules.
\vbox to 0.3cm{}

\leftskip0truemm \rightskip0truemm
\bigskip

Throughout this paper, all rings are commutative rings with $1\not= 0$ and all modules are unitary. Let $R$ be a ring. We denote by $\Nil(R)$  the nilpotent radical of $R$, $\Z(R)$ the set of all zero-divisors of $R$ and $\T(R)$ the localization of $R$ at the set of all regular elements. The $R$-submodules $I$ of $\T(R)$ such that $sI\subseteq R$ for some regular element $s$ are said to be \emph{fractional ideals}. Let $I$ be fractional ideal of $R$, we denote by  $I^{-1}=\{x\in T(R)|Ix\subseteq R\}$. An ideal $I$ of $R$ is said to be  nonnil provided that there is a non-nilpotent element in  $I$.  Denote by $\NN(R)$ the set of all nonnil ideals of $R$.

Recall from \cite{A97} that a ring $R$ is  an \emph{$\NP$-ring} if $\Nil(R)$ is a prime ideal, and a \emph{$\ZN$-ring} if $\Z(R)=\Nil(R)$. A prime ideal $P$ is said to be \emph{divided prime} if $P\subsetneq (x)$, for every $x\in R-P$. Set $\mathcal{H}=\{R|R$ is a commutative ring and \Nil(R)\ is a divided prime ideal of $R\}$. A ring $R$ is a \emph{$\phi$-ring} if $R\in \mathcal{H}$. Moreover, a $\ZN$ $\phi$-ring is said to be a \emph{strongly $\phi$-ring}. For a  $\phi$-ring $R$, there is a ring homomorphism $\phi:\T(R)\rightarrow R_{\Nil(R)}$ such that $\phi(a/b)=a/b$ where $a\in R$ and $b$ is a regular element. Denote by the ring $\phi(R)$ the image of $\phi$ restricted to $R$.    Let $R$  be an $\NP$-ring. It is easy to verify that $\NN(R)$ is a multiplicative system of ideals.  Let $M$ be an $R$-module.  Define
\begin{center}
$\phi$-$tor(M)=\{x\in M|Ix=0$ for some  $I\in \NN(R)\}$.
\end{center}
An $R$-module $M$ is said to be \emph{$\phi$-torsion} (resp., \emph{$\phi$-torsion free}) provided that  $\phi$-$tor(M)=M$ (resp., $\phi$-$tor(M)=0$).

In 2001, Badawi \cite{A01} investigated  $\phi$-chain rings ($\phi$-CRs for short) and $\phi$-pseudo-valuation rings as a $\phi$-version of chain rings and pseudo-valuation rings. In 2004, Anderson and Badawi \cite{FA04} introduced the concept of $\phi$-\Prufer\ rings and showed that a $\phi$-ring $R$ is $\phi$-\Prufer\ if and only if $R_{\fkm}$ is a $\phi$-chain ring for any maximal ideal $\fkm$ of $R$ if and only if $R/\Nil(R)$ is a \Prufer\ domain if and only if $\phi(R)$ is \Prufer. Recently, the authors in \cite{ZW21} said a  $\phi$-ring to be a \emph{$\phi$-\Prufer\ $v$-multiplication ring} ($\phi$-$\PvMR$ for short) for short provided that any finitely generated nonnil ideal is $\phi$-$w$-invertible, and showed that  a  $\phi$-ring $R$ is a $\phi$-$\PvMR$ if and only if    $R_{\fkm}$ is a $\phi$-CR for any $\fkm\in w$-$Max(R)$, if and only if    $R/\Nil(R)$ is a $\PvMD$, if and only if    $\phi(R)$ is a $\PvMR$. Let $R$ be an $\NP$-ring. The author in \cite{ZW21} called an $R$-module $M$ if $\Ext_R^1(T,M)$ is $\GV$-torsion for any $\phi$-torsion $R$-module $T$. Then they gave a homological characterization of $\phi$-$\PvMR$s in term of $\phi$-$w$-flat modules  when $R$ is a strongly $\phi$-ring.

In this paper, we study the ``hereditary'' version of $\phi$-$w$-flat modules over  $\NP$-rings $R$, that is, an $R$-module $M$ is said to be \emph{strongly $\phi$-$w$-flat} if $\Ext_R^n(T,M)$ is $\GV$-torsion for any $\phi$-torsion $R$-module $T$ and any $n\geq 1$. We show that all $\phi$-$w$-flat modules are strongly $\phi$-$w$-flat over $\ZN$-rings, and then induced the related homological theories. Finally, we characterize $\phi$-\Prufer\ $v$-multiplication strongly $\phi$-rings in terms of strongly $\phi$-$w$-flat modules.

As our work is related with $w$-operations, we give a quick review. Let $J$ a finitely generated ideal of a ring $R$. Then $J$ is called a \emph{$\GV$-ideal} if the natural homomorphism $R\rightarrow \Hom_R(J,R)$ is an isomorphism. The set of all $\GV$-ideals is denoted by $\GV(R)$. Certainly,  $\GV$-ideals are nonnil.
An $R$-module $M$ is said to be \emph{$\GV$-torsion} if for any $x\in M$ there is a $\GV$-ideal $J$ such that $Jx=0$; an $R$-module $M$ is said to be \emph{$\GV$-torsion free} if $Jx=0$, then $x=0$ for any $J\in\GV(R)$ and $x\in M$.  A $\GV$-torsion free module $M$ is said to be a \emph{$w$-module} if  for any $x\in E(M)$ there is a $\GV$-ideal $J$ such that $Jx\subseteq M$ where $E(M)$ is the injective envelope of $M$. The \emph{$w$-envelope} $M_w$ of a $\GV$-torsion free module $M$ is defined  by the minimal $w$-module  that contains $M$. Therefore, a $\GV$-torsion free module $M$ is a $w$-module if and only if $M_w=M$. A \emph{maximal $w$-ideal} for which is maximal among the $w$-submodules of $R$ is proved to be prime (see {\cite[Proposition 3.8]{hfxc11}}). The set of all maximal $w$-ideals is denoted by $w$-$\Max(R)$.  A sequence $A\rightarrow B\rightarrow C$ is said to be \emph{$w$-exact} if for any $ \p\in w$-$\Max(R)$, $A_\p\rightarrow B_\p\rightarrow C_\p$ is exact. A class $\mathcal{C}$ of $R$-modules is said to be \emph{closed under $w$-isomorphisms}  provided that for any $w$-isomorphism $f:M\rightarrow N$, if one of the modules $M$ and $N$ is in $\mathcal{C}$, so is the other. An $R$-module $M$ is said to be of \emph{finite type} if there exist a finitely generated free module $F$ and a $w$-epimorphism $g: F\rightarrow M$, or equivalently, if there exists a
finitely generated $R$-submodule $N$ of $M$ such that $N_w = M_w$. Certainly, the class of  finite type modules is closed under $w$-isomorphisms.

\section{strongly $\phi$-$w$-flat modules}

Let $R$ be a $\NP$-ring. Recall from  \cite{ZW21} called an $R$-module $M$ $\phi$-$w$-flat if $\Tor^R_1(T,M)$ is $\GV$-torsion for any $\phi$-torsion $R$-module $T$.  We begin with the definition of strongly $\phi$-$w$-flat modules.
\begin{definition}\label{w-phi-flat }
Let $R$ be a $\NP$-ring. An $R$-module $M$ is said to be \emph{strongly $\phi$-$w$-flat} if $\Tor^R_n(T,M)$ is $\GV$-torsion for any $\phi$-torsion $R$-module $T$ and any $n\geq 1$.
\end{definition}

 \begin{proposition}\label{w-phi-flat}
Let $R$ be an $\NP$-ring. The following statements are equivalent for an $R$-module $M$:
\begin{enumerate}
    \item $M$ is strongly $\phi$-$w$-flat;
    \item  $M_{\fkm}$ is strongly ${\phi}$-flat over $R_{\fkm}$ for all $\fkm\in w$-$Max(R)$;
        \item  $M_{\p}$ is strongly ${\phi}$-flat over $R_{\p}$ for all prime $w$-ideal $\p$ of $R$;
    \item $\Tor^R_n (T, M) $ is $\GV$-torsion for all $($finite type$)$ $\phi$-torsion $R$-modules $T$ for any $n\geq 1$;
    \item $\Tor^R_n (R/I, M)$ is $\GV$-torsion for all $($finite type$)$ nonnil ideals $I$ of $R$.
\end{enumerate}
\end{proposition}
\begin{proof} It follows by \cite[Theorem 6.2.15]{fk16}, \cite[Lemma 1.4]{ZW22-strong phi-dim}  and \cite[Theorem 3.2]{ZWT13}.
\end{proof}

 \begin{corollary}\label{w-closed}
Let $R$ be an $\NP$-ring and $0\rightarrow A\rightarrow B\rightarrow C\rightarrow 0$ be $w$-exact seqeunce. Then the following statements hold.
\begin{enumerate}
    \item The class of $\phi$-$w$-flat modules is closed under $w$-isomorphisms.
    \item If $B$ and $C$ are strongly $\phi$-$w$-flat, so is $A$.
    \item If $A$ and $C$ are strongly $\phi$-$w$-flat, so is $B$.
\end{enumerate}
\end{corollary}
\begin{proof} (1) Let $f:M\rightarrow N$ be a $w$-isomorphism and $T$  a $\phi$-torsion module.  There exist two exact sequences $0\rightarrow T_1\rightarrow M\rightarrow L\rightarrow 0$ and $0\rightarrow L\rightarrow N\rightarrow T_2\rightarrow 0$ with $T_1$ and $T_2$ $\GV$-torsion. For any $n\geq 1$,
consider the induced two long exact sequences $$\Tor^R_n(T,T_1)\rightarrow \Tor^R_n (T,M)\rightarrow \Tor^R_n (T,L)\rightarrow\Tor^R_{n-1}(T,T_1)$$ and $$ \Tor^R_{n+1} (T,T_2)\rightarrow \Tor^R_n(T,L)\rightarrow \Tor^R_n(T,N)\rightarrow \Tor^R_n(T,T_2).$$ Since $\Tor^R_n(T,T_i)$ and $\Tor^R_{n-1}(T,T_i)$ are $\GV$-torsion modules for any $i=1,2$ and $n\geq 1$,  we have $M$ is $\phi$-$w$-flat if and only if $N$ is $\phi$-$w$-flat by Proposition \ref{w-phi-flat}.

(2) and (3) can be similarly deduced as the classical cases.
\end{proof}

Recall from \cite{fk10} that a ring $R$ is said to be a \emph{$\DW$ ring} if every ideal of $R$ is a $w$-ideal. Then a ring $R$ is a $\DW$
ring if and only if every $R$-module is a $w$-module, if and only if $\GV(R) = \{R\}$ (see \cite[Theorem 3.8]{fk10}). Let $R$ be a $\NP$-ring. An $R$-module $M$ is said to be \emph{strongly $\phi$-flat} if $\Tor^R_n(T,M)=0$ for any $\phi$-torsion $R$-module $T$ and any $n\geq 1$.
Trivially, If $R$ is a $\DW$-ring, every  strongly $\phi$-$w$-flat module is strongly $\phi$-flat. Furthermore, this property characterizes $\DW$ rings.
\begin{proposition}\label{w-flat}
 The following statements are equivalent for an $\NP$-ring $R$:
\begin{enumerate}
    \item $R$ is a $\DW$ ring;
    \item  every strongly  $\phi$-$w$-flat module is strongly  $\phi$-flat;
    \item  every $w$-flat module is strongly  $\phi$-flat.
    \item  every $w$-flat module is  $\phi$-flat.
\end{enumerate}
\end{proposition}
\begin{proof} $(1)\Rightarrow (2)\Rightarrow (3)\Rightarrow (4)$: Trivial.

$(4)\Rightarrow (1)$:  Follows by \cite[Theorem 1.10]{ZW21}.
\end{proof}

\begin{remark}\label{flat-nots}
It follows by \cite[Theorem 1.4]{ZW21}, every strongly $\phi$-$w$-flat is $\phi$-$w$-flat. However, the converse is not true in general. Let $\mathbb{Z}$ be the ring of all integers with $\mathbb{Q}$ its quotients field, and  $\mathbb{Z}(\p^{\infty}):=\{\frac{n}{\p^k}+\mathbb{Z}\mid \frac{n}{\p^k}+\mathbb{Z} \in \mathbb{Q}/\mathbb{Z}\}$  the $\p$-Pr\"{u}fer group with $\p$ a prime in $\mathbb{Z}$.
Set $R=\mathbb{Z}(+)\mathbb{Z}(\p^{\infty})$ the trivial extension of $\mathbb{Z}$ with $\mathbb{Z}(\p^{\infty})$. Then $R$ is a $\phi$-ring with $\Nil(R)= 0(+)\mathbb{Z}(\p^{\infty})$.
Note that $R$ is a $\phi$-principal ideal ring (i.e., every nonnil ideal ideal is principal). Since every $\GV$-ideal is a nonnil ideal, and hence is principal. And so $R$ is a $\DW$-ring by \cite[Exercise 6.11(2)]{fk16}. Consequently, it follows from \cite[Example 1.1]{ZW22-strong phi-dim}  that $R/\Nil(R)$ is a  $\phi$-$w$-flat module which is  not strongly  $\phi$-flat, and hence not strongly  $\phi$-$w$-flat by Proposition \ref{w-flat}.

It follows from \cite[Theorem 1.9]{ZW21} that $\phi$-$w$-flat modules can be $w$-flat only over integral domains. This example also shows that strongly $\phi$-$w$-flat modules can be $w$-flat over non-integral domains(see \cite[Example 1.9]{ZW22-strong phi-dim}).
\end{remark}

If $R$ is an integral domain, every  strongly $\phi$-$w$-flat module is $w$-flat. Furthermore, every  strongly $\phi$-$w$-flat module is $w$-flat over ZN-rings.
\begin{theorem}\label{sp-ss}
Let the $\NP$-ring $R$ be  a \ZN\ ring.  Then  an $R$-module $M$ is  $\phi$-$w$-flat if and only if it is strongly $\phi$-$w$-flat.
\end{theorem}
\begin{proof} Let $I$ be a nonnil ideal of $R$. Since $R$ is a \ZN\ ring, then $I$ contains a nonzero-divisor $a$. Suppose $M$ is a $\phi$-$w$-flat $R$-module.  Since $a$ is a non-zero-divisor of $R$, $\Tor_n^R(R/\langle a\rangle,M)=0$ for any positive integer $n$.  It follows by \cite[Proposition 4.1.1]{CE56} that  $$\Tor_1^{R/\langle a\rangle}(R/I,M/aM)\cong \Tor_1^{R/\langle a\rangle}(R/I,M\otimes_RR/\langle a\rangle)\cong \Tor_1^{R}(R/I,M)$$ which is $\GV$-torsion. Hence $M/Ma$ is a $w$-flat $R/\langle a\rangle$-module. Consequently, for any $n\geq 1$ we have $$\Tor_n^R(R/I,M)\cong \Tor_n^{R/\langle a\rangle}(R/I,M\otimes_RR/\langle a\rangle)\cong \Tor_n^{R/\langle a\rangle}(R/I,M/aM)$$ is $\GV$-torsion by \cite[Theorem 6.7.2]{fk16}. It follows that $M$ is a strongly $\phi$-flat $R$-module.
\end{proof}

\begin{remark}  Recall from \cite{Z22} that an $R$-module $M$ is said to be regular $w$-flat if  $\Tor^R_1(R/I,M)$ is $\GV$-torsion for any regular ideal (i.e., an ideal that contains a non-zero-divisor) $I$ of $R$. Similar with the proof of Theorem \ref{sp-ss}, one can show that an  $R$-module $M$ is  regular $w$-flat if and only if  $\Tor^R_n(R/I,M)=0$ for any regular ideal $I$ of $R$ and any $n\geq 1$.
\end{remark}

%\begin{remark} The converse of Theorem \ref{sp-ss} does not hold in general. For example, let $R=\mathbb{Z}(+)\mathbb{Z}/2\mathbb{Z}$. Then $R$ is a principal ideal $\NP$-ring. So $R$ is a $\DW$-ring with $\Nil(R)=0(+)\mathbb{Z}/2\mathbb{Z}$. Note that the non-nilpotent zero-divisors in $R$ are of the form $(n,i)$ with $n$ a nonzero even number. Then we have the following free resolution of $R(n,i)$:$$\xrightarrow{d_{n+1}}R\xrightarrow{d_n}R\xrightarrow{d_{n-1}}\cdots\xrightarrow{d_2} R\xrightarrow{d_1} R\xrightarrow{\times (n,i)} R(n,i)\rightarrow 0,$$where $d_n$ is a multiple by $(0,1)$ when $n$ is an odd, and  $d_n$ is a multiple by $(2,1)$ when $n$ is an even. Now, let $I=R(a,b)$ be a nonnil ideal of $R$ with $a\not=0$. If $a$ is an odd, then $(a,b)$ is a non-zero-divisor, and so $\Tor_R^{\geq 2}(R/I,M)=0$. If $a$ is an even, then for any $n\geq 2$, $\Tor_R^{n}(R/I,M)\cong \Ker(d_{n}\otimes_RM)/\Im(d_{n+1}\otimes_RM)$\end{remark}

It is known that a \ZN\ $\phi$-ring is exactly a strongly $\phi$-ring. We consider the converse of Theorem \ref{sp-ss} under some assumptions.

\begin{theorem}\label{op-fn}
Let $R$ be  a $\phi$-ring with $(0:_Ra)$ finitely generated for any non-nilpotent element $a$ $($e.g. $R$ is a nonnil-coherent ring$)$ or $\Nil(R)$ nilpotent.  If every  $\phi$-$w$-flat  $R$-module is  strongly $\phi$-$w$-flat, then $R$ is a strongly $\phi$-ring.
\end{theorem}
\begin{proof} Let $R$ be  a $\phi$-ring and  $a$ a non-nilpotent element in $R$. It follows by \cite[Proposition 1.7]{ZW21} that $R/\Nil(R)$ is a $\phi$-$w$-flat  $R$-module, and so is strongly $\phi$-$w$-flat. Hence,  $$\Tor_2^R(R/Ra,R/\Nil(R))\cong\Tor_1^R(R/(0:_Ra),R/\Nil(R))\cong \frac{(0:_Ra)\cap\Nil(R)}{(0:_Ra)\Nil(R)} $$ is $\GV$-torsion. Since $R$ is  a $\phi$-ring,    $(0:_Ra)\subseteq \Nil(R)$, and so $(0:_Ra)\cap\Nil(R)=(0:_Ra)$. So  $\Tor_2^R(R/Ra,R/\Nil(R))\cong \frac{(0:_Ra)}{(0:_Ra)\Nil(R)}$.

(1) Suppose $(0:_Ra)$ is finitely generated. Then there exists $J\in \GV(R)$ such that $J\frac{(0:_Ra)}{(0:_Ra)\Nil(R)}=0$. So $J(0:_Ra)\subseteq (0:_Ra)\Nil(R)=(0:_Ra)J\Nil(R)$ by \cite[Lemma 1.6]{ZW21}. Hence $J(0:_Ra)=0$ by Nakayama's lemma. Since $J$ is semi-regular, so $(0:_Ra)=0$, that is, $a$ is a nonzero-divisor.

(2)  Suppose  $\Nil(R)$ is nilpotent. Let $\p$ be a prime $w$-ideal of $R$.  Then $$0=\Tor_2^R(R/Ra,R/\Nil(R))_\p\cong \frac{(0:_Ra)_\p}{(0:_Ra)_\p\Nil(R_\p)}\cong \frac{(0:_{R_\p}\frac{a}{1})}{(0:_{R_\p}\frac{a}{1})\Nil(R_\p)}.$$
It follows that $(0:_{R_\p}\frac{a}{1})=(0:_{R_\p}\frac{a}{1})\Nil(R_\p)$. Since $\Nil(R)$ is nilpotent, so is $\Nil(R_\p)$. Hence $(0:_{R_\p}\frac{a}{1})=(0:_Ra)_\p=0$ by  Nakayama's lemma. Hence $(0:_Ra)$ is $\GV$-torsion by \cite[Theorem 6.2.15]{fk16}. It follows by \cite[Proposition 6.1.20]{fk16} that $(0:_Ra)$ is a $w$-module, and so $(0:_Ra)=0$, that is, $a$ is a nonzero-divisor.
\end{proof}

\begin{remark}  We do not known that whether the condition ``$(0:_Ra)$ finitely generated for any non-nilpotent element $a$ or $\Nil(R)$ nilpotent'' in Theorem \ref{op-fn} can be removed, and so we propose the following conjecture:

Let $R$ be a $\phi$-ring. If  every  $\phi$-$w$-flat  $R$-module is  strongly $\phi$-$w$-flat, then
 $R$ is a strongly $\phi$-ring.
\end{remark}

\section{homological properties of strongly $\phi$-$w$-flat modules}

Let $R$ be a ring. It is well known that the flat dimension of an $R$-module $M$ is defined as the shortest flat resolution of $M$ and the weak global dimension of $R$ is the supremum of the flat dimensions  of all $R$-modules. The $w$-flat dimension $w$-fd$_R(M)$ of an $R$-module $M$ and $w$-weak global dimension $w$-w.gl.dim$(R)$ of a ring $R$ were introduced and studied in \cite{fq15}.  We now introduce the notion of $\phi$-$w$-flat dimension of an $R$-module as follows.

\begin{definition}\label{w-phi-flat }
Let $R$ be a ring and $M$ an $R$-module. We write $\phi_s$-$w$-fd$_R(M)\leq n$  ($\phi_s$-$w$-fd abbreviates  \emph{strongly $\phi$-$w$-flat dimension}) if there is a $w$-exact sequence of $R$-modules
$$ 0 \rightarrow F_n \rightarrow \cdots\rightarrow F_1\rightarrow F_0\rightarrow M\rightarrow 0   \ \ \ \ \ \ \ \ \ \ \ \ \ \ \ \ \ \ \ \ \ \ \ \ \ \ \ \ \ \ \ \ \ \ \ \ \ \ \ (\diamondsuit)$$
where each $F_i$ is  strongly $\phi$-$w$-flat for $i=0,\dots,n$. The $w$-exact sequence $(\diamondsuit)$ is said to be  a $\phi_s$-$w$-flat $w$-resolution of length $n$ of $M$. If such finite $w$-resolution does not exist, then we say $\phi_s$-$w$-fd$_R(M)=\infty$; otherwise,  define $\phi_s$-$w$-fd$_R(M) = n$ if $n$ is the length of the shortest $\phi_s$-$w$-flat $w$-resolution of $M$.
\end{definition}\label{def-wML}
It is obvious that an $R$-module $M$ is strongly $\phi$-$w$-flat if and only if $\phi_s$-$w$-fd$_R(M)= 0$. Certainly, $\phi_s$-$w$-fd$_R(M)\leq w$-fd$_R(M)$.

\begin{proposition}\label{w-phi-flat d}
Let $R$ be an $\NP$-ring. The following statements are equivalent for an $R$-module $M$:
\begin{enumerate}
    \item $\phi_s$-$w$-{\rm fd}$_R(M)\leq n$;
    \item $\Tor^R_{n+k}(T, M)$ is $\GV$-torsion for all $\phi$-torsion $R$-modules $T$ and all positive integer $k$;
        \item $\Tor^R_{n+k}(R/I,M)$  is $\GV$-torsion for all nonnil ideals $I$ and all positive integer $k$;
    \item $\Tor^R_{n+k}(R/I,M)$  is $\GV$-torsion for all finitely generated nonnil ideals $I$ and all positive integer $k$;
    \item if $0 \rightarrow F_n \rightarrow \cdots \rightarrow F_1\rightarrow F_0\rightarrow M\rightarrow 0$ is an $($a $w$-$)$exact sequence, where $F_0, F_1,\dots, F_{n-1}$ are strongly $\phi$-$w$-flat $R$-modules, then $F_n$ is strongly $\phi$-$w$-flat;
    \item if $0 \rightarrow F_n \rightarrow \cdots \rightarrow F_1\rightarrow F_0\rightarrow M\rightarrow 0$ is an $($a $w$-$)$exact sequence, where $F_0, F_1,\dots, F_{n-1}$ are $w$-flat $R$-modules, then $F_n$ is strongly $\phi$-$w$-flat;
    \item if $0 \rightarrow F_n \rightarrow \cdots \rightarrow F_1\rightarrow F_0\rightarrow M\rightarrow 0$ is an $($a $w$-$)$exact sequence, where $F_0, F_1,\dots, F_{n-1}$ are flat $R$-modules, then $F_n$ is strongly $\phi$-$w$-flat;
    \item there exists  $($a $w$-exact$)$ an exact sequence  $0 \rightarrow F_n \rightarrow \cdots \rightarrow F_1\rightarrow F_0\rightarrow M\rightarrow 0$ ,
where $F_0, F_1,\dots, F_{n-1}$ are $w$-flat $R$-modules, then $F_n$ is strongly $\phi$-$w$-flat;
 \item there exists  $($a $w$-exact$)$ an exact sequence  $0 \rightarrow F_n \rightarrow \cdots \rightarrow F_1\rightarrow F_0\rightarrow M\rightarrow 0$ ,
where $F_0, F_1,\dots, F_{n-1}$ are flat $R$-modules, then $F_n$ is strongly $\phi$-$w$-flat.
\end{enumerate}
\end{proposition}
\begin{proof}
$(1) \Rightarrow(2)$: We prove $(2)$ by induction on $n$.  For the case $n = 0$, (2) trivially holds  as $M$ is strongly $\phi$-$w$-flat. If $n>0$, then
there is a  $w$-exact sequence  $0 \rightarrow F_n \rightarrow \cdots \rightarrow F_1\rightarrow F_0\rightarrow M\rightarrow 0$,
where each $F_i$ is strongly $\phi$-$w$-flat for $i=0,\dots,n$. Set $K_0 = \ker(F_0\rightarrow M)$. Then both
$0 \rightarrow  K_0 \rightarrow  F_0 \rightarrow  M \rightarrow  0 $ and $0 \rightarrow  F_n \rightarrow  F_{n-1} \rightarrow\cdots \rightarrow  F_1 \rightarrow  K_0 \rightarrow  0$ are $w$-exact. So $\phi$-fd$_R(K_0)\leq n-1$. By induction, $\Tor^R_{n-1+k}(T, K_0)$ is $\GV$-torsion
for all $\phi$-torsion $R$-modules $T$ and all positive integer $k$. By \cite[Theorem 6.7.2]{fk16}, it follows from the $w$-exact sequence   $$0=\Tor^R_{n+k}(T,F_0 )\rightarrow\Tor^R_{n+k}(T,M )\rightarrow \Tor^R_{n-1+k}(T, K_0 )\rightarrow \Tor^R_{n-1+k}(T, F_0)$$  that
$\Tor^R_{n+k}(T, M)$ is $w$-isomorphic to $\Tor^R_{n-1+k}(T, K_0 )$ which is $\GV$-torsion for  all $\phi$-torsion $R$-modules $T$ and all positive integer $k$.

$(2) \Rightarrow(3)\Rightarrow(4)$ and $(5) \Rightarrow(6)\Rightarrow(7)$:  Trivial.

$(4) \Rightarrow(5)$: Let $K_0 = \ker(F_0 \rightarrow  M)$ and $K_i = \ker(F_i \rightarrow  F_{i-1})$, where
$i = 1,\dots, n-1$. Then $K_{n-1}$ is $w$-isomorphic to $F_n$. Since all $F_0, F_1,\dots, F_{n-1}$ are strongly $\phi$-$w$-flat,
$\Tor^R_k (R/I, F_n)$  is $w$-isomorphic to $ \Tor^R_{n+k}(R/I, M)$ which  is $\GV$-torsion   for all finitely generated nonnil ideal $I$ and any  positive integer $k$   by \cite[Theorem 6.7.2]{fk16}. Hence $F_n$
is strongly $\phi$-$w$-flat by Proposition \ref{w-phi-flat}.

$(6) \Rightarrow (8)$: It follows by \cite[Theorem 3.5]{Z19-w-ce} that the class of $w$-flat modules is covering, we can construct an exact sequence $\cdots \rightarrow F_{n-1} \xrightarrow{d_{n-1}} F_{n-2}\rightarrow  \cdots \rightarrow F_1\rightarrow F_0\rightarrow M\rightarrow 0$, where $F_0, F_1,\dots, F_{n-1}$ are $w$-flat $R$-modules, then $F_n:=\Ker(d_{n-1})$ is strongly $\phi$-$w$-flat by $(6)$.

$(7) \Rightarrow (9)$: The proof is similar with that of $(6) \Rightarrow (8)$.

$(9) \Rightarrow (8) \Rightarrow (1)$: Trivial.
\end{proof}

\begin{proposition}\label{spd-ext}
Let $R$ be an $\NP$-ring and $0\rightarrow A\rightarrow B\rightarrow C\rightarrow 0$ be a short $w$-exact sequence of $R$-modules. Then the following statements hold.
\begin{enumerate}
  \item $\phi_s$-$w$-{\rm fd}$_R(C)\leq 1+\max\{\phi_s$-$w$-{\rm fd}$_R(A)$,
  $\phi_s$-$w$-{\rm fd}$_R(B)\}$.
    \item  If  $\phi_s$-$w$-{\rm fd}$_R(B)< \phi_s$-$w$-{\rm fd}$_R(C)$, then $\phi_s$-$w$-{\rm fd}$_R(A)= \phi_s$-$w$-{\rm fd}$_R(C)-1> \phi_s$-$w$-{\rm fd}$_R(B).$
\end{enumerate}
\end{proposition}
\begin{proof}

$(1)$ Assume the right side of $(1)$ is finite. Suppose  $\phi_s$-$w$-{\rm fd}$_R(A)\leq n$, $\phi_s$-$w$-{\rm fd}$_R(B)\leq n$.
For any $\phi$-torsion module $T$, it follows by \cite[Theorem 6.7.2]{fk16} that we have the following $w$-exact sequence:
$$\Tor^R_{n+2}(B,T)\rightarrow \Tor^R_{n+2}(C,T) \rightarrow \Tor^R_{n+1}(A,T).$$
By Proposition \ref{w-phi-flat d}, $\Tor^R_{n+2}(B,T)$ and $\Tor^R_{n+1}(A,T)$ is $\GV$-torsion,
and so is $\Tor^R_{n+2}(C,T)$. Hence $\phi_s$-$w$-{\rm fd}$_R(C)\leq 1+n$.

$(2)$ Assume $\phi_s$-$w$-{\rm fd}$_R(B)=n$. Suppose $k>n$. For any $\phi$-torsion module $T$,  it follows by \cite[Theorem 6.7.2]{fk16} that we have the following $w$-exact sequence:
$$\Tor^R_{k+1}(B,T) \rightarrow \Tor^R_{k+1}(C,T) \rightarrow \Tor^R_{k}(A,T) \rightarrow \Tor^R_{k}(B,T).$$
By Proposition \ref{w-phi-flat d},  $\Tor^R_{k+1}(B,T)$ and $\Tor^R_{k}(B,T)$ is $\GV$-torsion,
we have $\Tor^R_k(A,T)$ is $\GV$-torsion if and only if $\Tor^R_{k+1}(C,T)$ is $\GV$-torsion.
So $\phi_s$-$w$-{\rm fd}$_R(C)< \infty$ if and only if $\phi_s$-$w$-{\rm fd}$_R(A)<\infty$.

Assume $\phi_s$-$w$-{\rm fd}$_R(A)=s$ and $\phi_s$-$w$-{\rm fd}$_R(C)=m$.  For any $\phi$-torsion module $T$, we have $\Tor^R_{s+1}(A,T)$ and $\Tor^R_{s+2}(C,T)$ is $\GV$-torsion,  and so $m\leq s+1$.
On the other hand, since $\Tor^R_{m+1}(C,T)$ and $\Tor^R_{m}(A,T)$ is $\GV$-torsion,
so $s\leq m-1$, and hence $s=m-1$.
\end{proof}

Now, we are ready to introduce the $\phi$-$w$-weak global dimension of a ring  in terms of $\phi$-flat dimensions.
\begin{definition}\label{w-phi-flat }
The \emph{$\phi$-$w$-weak global dimension} of a ring $R$ is defined by
\begin{center}
$\phi$-$w$-w.gl.dim$(R) = \sup\{\phi_s$-$w$-{\rm fd}$_R(M) | M $ is an $R$-module$\}.$
\end{center}
\end{definition}\label{def-wML}

Obviously, by definition, $\phi$-$w$-w.gl.dim$(R)\leq w$-w.gl.dim$(R)$.  Notice that if $R$ is an integral domain, then $\phi$-$w$-w.gl.dim$(R)= w$-w.gl.dim$(R)$. The following result can be deduced easily by Proposition \ref{w-phi-flat d}.

\begin{theorem}\label{w-g-flat}
Let $R$ be an $\NP$-ring. The following statements are equivalent for $R$.
\begin{enumerate}
   \item  $\phi$-$w$-{\rm w.gl.dim}$(R)\leq  n$.

    \item $\Tor^R_{n+k}(M, N)$ is $\GV$-torsion for all $R$-modules $M$ and $\phi$-torsion $N$ and all $k > 0$.
    \item  $\Tor^R_{n+1}(M, N)$ is $\GV$-torsion for all $R$-modules $M$ and $\phi$-torsion $N$.
     \item  $\Tor^R_{n+1}(R/I,M)$ is $\GV$-torsion for all $R$-modules $M$ and (finite type) nonnil ideals $I$ of $R$.
    \item  $\Tor^R_{n+k}(R/I,M)$ is $\GV$-torsion for all $R$-modules $M$ and(finite type)  nonnil ideals $I$ of $R$ and all positive integer $k$.
    \item  $\phi_s$-$w$-{\rm fd}$_R(M)\leq n$ for all $R$-modules $M$.
    \item $\phi_s$-$w$-{\rm fd}$_R(R/I)\leq n$ for all  nonnil ideals $I$ of $R$.
    \item $\phi_s$-$w$-{\rm fd}$_R(R/I)\leq  n$ for all finite type  nonnil ideals $I$ of $R$.

\end{enumerate}
Consequently, the $\phi$-$w$-weak global dimension of $R$ is determined by the
formulas:
\begin{align*}
\phi\mbox{-}w\mbox{-{\rm w.gl.dim}}(R)&= \sup \{w\mbox{-}{\rm fd}_R(R/I) |\ I\ is\ a\  nonnil\ ideal\ of\ R\}\\
&= \sup \{w\mbox{-}{\rm fd}_R(R/I) |\ I\ is\ a\ finite\ type\ nonnil\ ideal\ of\ R\} .
\end{align*}
\end{theorem}

\begin{proposition}\label{ideal-1}
Let $D$ be an integral domain, $Q$ its quotient field and $V$ a $Q$-linear space. Then $\phi$-$w$-{\rm w.gl.dim}$(D(+)V)=w$-{\rm w.gl.dim}$(D)$.
\end{proposition}
\begin{proof}  Set $R=D(+)V$. Assume $w$-{\rm w.gl.dim}$(D)\leq n$. Let $M$ be an $R$-module, Then $M$ is naturally an $D$-module. Let $J$ be an nonnil ideal of $R$. Then  $J=I(+)V$ with $I$ a nonzero $D$-ideal. Note that  $R$ is a flat $D$-module. By \cite[Proposition 4.1.2]{CE56} we have $$\Tor_{n+1}^R(R/J,M)\cong\Tor_{n+1}^R(D/I\otimes_DR,M)\cong \Tor_{n+1}^D(D/I,M)$$ is $\GV$-torsion $D$-module. By \cite[Proposition 1.2]{CK17}, it is also a $\GV$-torsion $R$-module. So $\phi$-$w$-{\rm w.gl.dim}$(D(+)V)\leq w$-{\rm w.gl.dim}$(D)$.

On the other hand, assume $\phi$-$w$-{\rm w.gl.dim}$(R)\leq m$. Suppose $N$ is an $D$-module. For any $(a,b)\in R$ and $t\in N$, Define $(a,b)t=at$. Then $N$ is naturally an $R$-module. Suppose $I$ is a nonzero $D$-ideal, then $J=I(+)V$ is an  nonnil ideal of $R$. By \cite[Proposition 4.1.2]{CE56} again,
 $$\Tor_{m+1}^D(D/I,N)\cong\Tor_{m+1}^R(D/I\otimes_DR,N) \cong \Tor_{m+1}^R(R/J,N)=0$$   is $\GV$-torsion $R$-module. By \cite[Proposition 1.2]{CK17}, it is also a $\GV$-torsion $D$-module.
So $\phi$-$w$-{\rm w.gl.dim}$(D(+)V)\geq w$-{\rm w.gl.dim}$(D)$. Consequently $\phi$-$w$-{\rm w.gl.dim}$(D(+)V)=w$-{\rm w.gl.dim}$(D)$.
\end{proof}

Recall from  \cite{ZWT13} that a $\phi$-ring $R$ is said to be \emph{$\phi$-von Neumann regular} provided that every $R$-module is $\phi$-flat. A $\phi$-ring $R$  is  $\phi$-von Neumann regular, if and only if there is an element $x\in R$ such that $a = xa^2$ for any non-nilpotent element $a\in R$, if and only if $R/\Nil(R)$ is a von Neumann regular ring, if and only if $R$ is zero-dimensional (see \cite[Theorem 4.1]{ZWT13}). Now, we give some more characterizations of $\phi$-von Neumann regular rings.

\begin{theorem}\label{w-g-flat-1}
Let $R$ be a $\phi$-ring. The following statements are equivalent for $R$:
\begin{enumerate}
    \item $\phi$-$w$-{\rm w.gl.dim}$(R)=0$;
    \item  $R$  is  $\phi$-von Neumann regular;
    \item  $a\in (a^2)_w$ for any non-nilpotent element $a\in R$;
    \item $w$-$dim(R)=0$ ;
    \item $dim(R)=0$;
     \item every $R$-module is $\phi$-$w$-flat;
    \item every $R$-module is strongly $\phi$-$w$-flat;
\end{enumerate}
\end{theorem}
\begin{proof}
The equivalence of $(1)-(6)$ follows from \cite[Theorem 3.1]{ZW21}.  $(1)\Leftrightarrow (7)$ is trivial.
\end{proof}

Recall that a ring $R$ is said to be a $\PvMR$ if  every finitely generated regular ideal $I$ is $w$-invertible, i.e., $(II^{-1})_w=R$,  or equivalently, there is a fractional ideal $J$ of $R$ such that $(IJ)_w=R$.  The author in \cite{Z22} give a homological characterization of $\PvMR$s, that is, a ring   $R$ is a $\PvMR$ if and only if  every finitely generated regular ideal is $w$-projective, if and only if   any submodule of a regular $w$-flat module is  regular $w$-flat, if and only if    any regular ideal of $R$ is  $w$-flat, if and only if     $reg$-$w$-w.gl.dim$(R)\leq 1$.

Let $R$ be a $\phi$-ring. Following \cite{A01}, a $\phi$-ring  $R$ is said to be a \emph{$\phi$-chain ring} ($\phi$-CR for short) if for any $a,b\in R-\Nil(R)$, either $a|b$ or $b|a$ in $R$.  Recall from \cite{kf12} that a nonnil ideal $J$ of $R$ is said to be a \emph{$\phi$-$\GV$-ideal} (resp., \emph{$\phi$-$w$-ideal}) of $R$ if $\phi(J)$ is a $\GV$-ideal (resp., $w$-ideal) of $\phi(R)$. A $\phi$-ring $R$ is called a \emph{$\phi$-\SM\ ring} if it satisfies the ACC on $\phi$-$w$-ideals. An ideal $I$ of $R$ is  \emph{$\phi$-$w$-invertible} if $(\phi(I)\phi(I)^{-1})_W=\phi(R)$ where $W$ is the $w$-operation of $\phi(R)$. Recall from \cite[Definition 3.2]{ZW21} that a $\phi$-ring $R$ is said to be a \emph{$\phi$-\Prufer\ $v$-multiplication ring} ($\phi$-$\PvMR$ for short) provided that any finitely generated nonnil ideal is $\phi$-$w$-invertible. It was proved in \cite[Theorem 3.3]{ZW21} that a  $\phi$-ring $R$ is a $\phi$-$\PvMR$ if and only if    $R_{\fkm}$ is a $\phi$-CR for any $\fkm\in w$-$Max(R)$, if and only if    $R/\Nil(R)$ is a $\PvMD$, if and only if    $\phi(R)$ is a $\PvMR$. The author in \cite{ZW21} also gave a  homological characterization of $\phi$-$\PvMR$s in term of $\phi$-$w$-flat modules when $R$ is a strongly $\phi$-ring. Now we characterize $\phi$-\Prufer\ multiplication strongly $\phi$-rings in terms of strongly $\phi$-$w$-flat modules.

\begin{theorem}\label{w-g-flat-1}
Let $R$ be a  $\phi$-ring. The following statements are equivalent for $R$:
\begin{enumerate}
   \item $R$ is a strongly $\phi$-ring and a  $\phi$-$\PvMR$;
   \item  $R$ is a $\phi$-$w$-w.gl.dim$(R)\leq 1$;
   \item every submodule of a  strongly  $\phi$-$w$-flat module is strongly  $\phi$-$w$-flat;
   \item every submodule of a $w$-flat module is strongly  $\phi$-$w$-flat;
   \item every submodule of a flat module is  strongly  $\phi$-$w$-flat;
   \item every  ideal of $R$ is  strongly  $\phi$-$w$-flat;
   \item every  nonnil ideal of $R$ is  $w$-flat;
   \item  every finite type nonnil ideal of $R$ is $w$-flat.
\end{enumerate}
\end{theorem}\

\begin{proof} Let $R$ be a  $\phi$-ring. Note that $(1)\Rightarrow (2)$ and $(8)\Rightarrow (2)$  follows by \cite[Theorem 3.3]{ZW21}.

$(2)\Leftrightarrow (3)\Rightarrow (4)\Rightarrow (5)\Rightarrow (6)\Rightarrow (7)\Rightarrow (8)$ is trivial.

$(8)\Rightarrow (1)$ It follows by  \cite[Theorem 3.3]{ZW21} that  we only need to show $R$ is a strongly $\phi$-ring. Let  $a$ be a zero-divisor in  $R$. Then there exists $0\not=r\in R$ satisfying $ar=0$. On contrary, suppose $a$ is not nilpotent. Since $Rr$ is $w$-flat, we have $\Tor_1^R(R/Ra,Rr)\cong \{xr\in Rr\mid axr = 0\}/(0:_Ra)Rr$ is $\GV$-torsion (see \cite[Proposition 1]{H60}). Since $ar=0$, there exists $J\in \GV(R)$ such that $Jr\subseteq (0:_Ra)Rr$. Since $J$ is finitely generated, there exists a finitely generated $R$-ideal $B\subseteq  (0:_RRa)$ satisfying $Jr\subseteq RBr$.
Since $R$ is a $\phi$-ring and $a$ is not nilpotent, we have $B$ is a nilpotent ideal. Assume that $B^m=0$. Then $$J^mr=J^{m-1}(Jr)\subseteq J^{m-2}J(RBr)\subseteq J^{m-3}J(RB^2r)\subseteq \cdots \subseteq J(RB^{m-1}r)\subseteq RB^mr=0.$$
Note that $J^m$ is a $\GV$-ideal, and so $(0:_RJ^m)=0$. Hence $r=0$, which is a contradiction.
\end{proof}

\begin{remark}\label{infty example}
Let $R$ be the ring in Example \ref{flat-nots}.  Then $R$ is both a $\DW$-ring and a $\phi$-\Prufer\ ring. So $R$ is a  $\phi$-$\PvMR$s.
It follows by \cite[Example 2.11]{ZW22-strong phi-dim} that  the $\phi$-weak global dimension of $R$ is also infinite. And so $\phi$-$w$-{\rm w.gl.dim}$(R)=\infty$. Note that the $\phi$-ring $R$ is not a strongly  $\phi$-ring. At last, we propose the following conjecture:

\begin{question}
    If $R$ is a strongly $\phi$-ring, then
    \begin{center}
    $\phi$-$w$-{\rm w.gl.dim}$(R)=w$-{\rm w.gl.dim}$(R/\Nil(R))$.
    \end{center}
\end{question}
\end{remark}

\begin{acknowledgement}\quad\\
The first author was supported by the National Natural Science Foundation of China (No. 12201361). The third author was supported by the National Natural Science Foundation of China (No. 12061001)
\end{acknowledgement}


\bigskip
\begin{thebibliography}{99}
\bibitem{AW09} D. D. Anderson and M. Winders, Idealization of a module. J. Commut. Algebra 1 (2009), no. 1, 3-56.
\bibitem{FA04} D. F. Anderson, A. Badawi, {\it On $\phi$-Pr\"{u}fer rings and $\phi$-Bezout rings},  Houston J. Math., \textbf{30} (2004), 331-343.

\bibitem{FA05}D. F. Anderson, A. Badawi,  {\it  On $\phi$-Dedekind rings and $\phi$-\Krull\ rings},  Houston J. Math., \textbf{31}  (2005), 1007-1022.

\bibitem{A97} A. Badawi,  {\it On divided commutative rings},  Comm. Algebra, \textbf{27} (1999), 1465-1474.

\bibitem{A01} A. Badawi, {\it On $\phi$-chained rings and $\phi$-pseudo-valuation rings},  Houston J. Math., \textbf{27} (2001), 725-736.

\bibitem{ALT06} A. Badawi,  T.  Lucas,  {\it On $\phi$-Mori rings},  Houston J. Math.,  \textbf{32} (2006), 1-32.

\bibitem{CE56} H. Cartan, S. Eilenberg,  {\it Homological algebra} Princeton University Press, Princeton, 1956.

\bibitem{CK17}  G. W. Chang and H. Kim, The $w$-FF property in trivial extensions, Bull. Iranian Math.
Soc. 43 (2017), no. 7, 2259-2267.
\bibitem{H60}   A. A. Hattori,  {\it Foundation of torsion theory for modules over general rings}, Nagoya
Math. J.,  \textbf{17} (1960) 147-158.

\bibitem{H88}  J. A.  Huckaba, {\it Commutative rings with Zero Divisors}, Monographs and Textbooks in Pure and Applied Mathematics, {\bf 117}, Marcel Dekker, Inc., New York, 1988.


\bibitem{fk14}  H. Kim, F. G. Wang,  {\it  On $LCM$-stable modules},  J. Algebra Appl., {\bf 13} (2014),  no. 4, 1350133, 18 p.

\bibitem{kf12}  H. Kim, F. G. Wang,   {\it On $\phi$-strong Mori rings},  Houston J. Math., {\bf 38} (2012), no. 2,  359-371.

\bibitem{QW15} L. Qiao, F. G. Wang,   {\it The $w$-weak global dimension of commutative rings}, Bull. Korean Math. Soc. 52 (2015), 1327-1338.

\bibitem{S79}  B. Stenstr\"{o}m,  {\it  Rings of Quotients}, Die Grundlehren Der Mathematischen Wissenschaften, Berlin: Springer-verlag, 1975.

\bibitem{f97} F. G. Wang,  {\it On $w$-projective modules and $w$-flat modules},  Algebra Colloq., {\bf 4} (1997), no. 1, 111-120.

\bibitem{fk10} F. G. Wang, {\it Finitely presented type modules and $w$-coherent rings},  J. Sichuan Normal Univ., {\bf 33} (2010), 1-9.

\bibitem{KW14} F. G. Wang,  H. Kim,  {\it $w$-injective modules and $w$-semi-hereditary rings},  J. Korean Math. Soc., {\bf 51} (2014),  no. 3, 509-525.

\bibitem{WK15} F. G. Wang, H. Kim,   {\it Two generalizations of projective modules and their applications},  J. Pure Appl. Algebra, {\bf 219} (2015), no. 6,  2099-2123.

\bibitem{fk16} F. G. Wang,  H. Kim,  {\it  Foundations of Commutative rings and Their Modules}, Singapore: Springer, 2016.

\bibitem{fm97} F. G. Wang,  R. L. McCasland,  {\it On $w$-modules over strong Mori domains},  Comm. Algebra, {\bf 25)}  (1997), no. 4, 1285-1306.

\bibitem{fq15} F. G. Wang, L. Qiao,    {\it  The $w$-weak global dimension of commutative rings}, Bull. Korean Math. Soc., {\bf 52} (2015), no. 4, 1327-1338.

\bibitem{fz18} F. G. Wang, D. C. Zhou,  {\it  A homological characterization of Krull domains}, Bull. Korean Math. Soc., {\bf 55} (2018),  649-657.

\bibitem{hfxc11}  H. Y. Yin, F. G. Wang,  X. S. Zhu and  Y. H. Chen, {\it $w$-modules over commutative rings},  J. Korean Math. Soc., {\bf 48} (2011), no. 1, 207-222.

\bibitem{ZWQ20} X. L. Zhang,  F. G. Wang and W. Qi,  {\it On characterizations of $w$-coherent rings},   Comm. Algebra,  {\bf 48} (2020), no. 11, 4681-4697.

\bibitem{Z22} X. L. Zhang, A homological characterization of Prufer $v$-multiplication rings, Bull. Korean Math. Soc., {\bf 59} (1) 2022, 213-226


\bibitem{ZW21}  X. L. Zhang, W. Zhao, On $\phi$-$w$-Flat modules and Their Homological Dimensions, Bull. Korean Math. Soc., {\bf 58}(2021), no. 5,1039-1052.
\bibitem{ZZW21} X. L. Zhang, W. Zhao and F. G. Wang, {\it  On $\phi$-flat cotorsion theory},  J. Guangxi Normal Univ.,  {\bf 39} (2021),  no. 2, 120-124.


\bibitem{ZW22-strong phi-dim}  X. L. Zhang, Strongly $\phi$-flat modules, strongly nonnil-injective modules and their homological dimensions, https://arxiv.org/abs/2211.14681

\bibitem{Z19-w-ce}  X. L. Zhang, Covering and Enveloping on $w$-operation, Journal of Sichuan Normal University (Natural Science) ,2019,42(03):382-386. (in Chinese)


\bibitem{Z18} W. Zhao,   {\it On $\phi$-flat modules and $\phi$-\Prufer\ rings},   J. Korean Math. Soc.,  {\bf 55} (2018), no. 5, 1221-1233.

\bibitem{ZWT13} W. Zhao,  F. G. Wang and G. H. Tang,   {\it On $\phi$-von Neumann regular rings},  J. Korean Math. Soc., {\bf 50} (2013), no. 1, 219-229.
\end{thebibliography}
\end{document}